\documentclass[reqno,12pt]{amsart}
\usepackage[letterpaper,hmargin=1in,vmargin=1in]{geometry}
\usepackage{mathrsfs, amsmath, amssymb, amsthm, verbatim, bbm} 
\usepackage[hidelinks]{hyperref}
\usepackage{comment}
\usepackage{graphicx}
\usepackage{xcolor,mathabx}
\usepackage{mathrsfs}
\usepackage{amsmath,amsfonts,verbatim}
\usepackage{latexsym}
\usepackage{amssymb,leftidx}
\usepackage{extarrows}
\usepackage{overpic}
\usepackage{color}
\usepackage{epsfig}
\usepackage{subfigure}
\usepackage{tikz}
\usepackage{mathtools}
\usepackage{graphicx}
\usepackage{float}
\usepackage{mathrsfs}
\usepackage{amscd}
\usepackage{bbm}
\usepackage{cite}
\newtheorem{lem}{Lemma}

\newcommand{\Real}{\mathbb{R}}

\newcommand{\CI}{\mathcal{C}^{\infty}}

\DeclareMathOperator{\supp}{\mathrm{supp}}
\numberwithin{equation}{section}
\numberwithin{lem}{section}

\newcommand{\norm}[2][]{\left\|#2\right\|_{#1}}

\newcommand\R{\mathbb{R}}
\newcommand\C{\mathbb{C}}

\newcommand{\RR}{\mathbb{R}}

\newcommand\al{\alpha}
\newcommand{\A}{\vec{A}}

\renewcommand{\Im}{\operatorname{Im}}
\renewcommand{\Re}{\operatorname{Re}}

\newcommand{\NN}{\mathbb{N}}
\newcommand{\dom}{\mathcal{D}}

\newcommand{\tP}{\widetilde{P}}

\newcommand{\reals}{\mathbb{R}}

\numberwithin{equation}{section}
\newtheorem{theorem}{Theorem}[section]
\newtheorem{proposition}[theorem]{Proposition}

\newtheorem{corollary}[theorem]{Corollary}
\newtheorem{remark}[theorem]{Remark}

\title[AB Strichartz]{Strichartz estimates for Schr\"odinger equations with the multipole Aharonov--Bohm Hamiltonian} 

\author{Mengxuan Yang}
\email{yangmx@princeton.edu}
\address{Operations Research \& Financial Engineering, Princeton University, Princeton, NJ, 08544, USA}

\author{Junyong Zhang}
\email{zhang\_junyong@bit.edu.cn}
\address{Department of Mathematics, Beijing Institute of Technology, Beijing, 100081, China}

\begin{document}

\begin{abstract}
We prove global-in-time Strichartz estimates for Schr\"odinger equations with multipole Aharonov--Bohm Hamiltonians on $\RR^2$. As intermediate steps, we prove global-in-time local smoothing estimates for multipole Aharonov--Bohm Hamiltonians. 
\end{abstract}

\maketitle

\section{Introduction}
We consider the Aharonov--Bohm Hamiltonian with multiple poles 
\begin{gather}\label{oper-P} 
     P \coloneqq (-i \vec \nabla - \vec A)^2, \qquad \vec A = \sum_{k=1}^N \alpha_k \vec A_0(x-x_k,y-y_k), \qquad  \vec A_0(x,y) = \frac{(-y,x)}{x^2+y^2}
     \end{gather}
where $\alpha_k, x_k, y_k\in \Real$ with $\alpha_k\notin \mathbb{Z}$. Let $s_k = (x_k,y_k)$ and $S = \{s_1,\dots,s_N\}$ be distinct poles of the vector potential $\Vec{A}$. In addition, we assume that there are no three poles collinear (cf.~Remark~\ref{rm:pole}). We equip the Hamiltonian $P$ with its Friedrichs domain $\mathcal D$, and we assume for convenience that $s_1$ is at the origin. This model corresponds to the Aharonov--Bohm Hamiltonian \cite{AB59} of $N$ infinitely thin solenoids in $\R^3$ parallel to the $z$-axis. In particular, while the magnetic potential is non-vanishing everywhere, the magnetic field $\vec{B}=\nabla\times\A$ vanishes in $\R^2\backslash S$. We prove global-in-time Strichartz estimates for Schr\"odinger equations with the Aharonov--Bohm Hamiltonian \eqref{oper-P}. 

\begin{theorem}\label{thm:stri-S} Let $P$ be the Friedrichs self-adjoint extension of the Schr\"odinger operator given in \eqref{oper-P} with $N\ge 1$. If $u$ solves the equation 
\begin{equation}\label{equ:S}
\begin{cases}
(D_t- P )u=0,\\
u|_{t=0}=f(x),
\end{cases}
\end{equation} 
where $D_t = \frac{1}{i}\partial_t$, then there exists a constant $C>0$ such that
\begin{equation}\label{str-S}
\|u(t,x)\|_{L^p_t(\R;L^q(\R^2))}\leq C \|f\|_{L^2(\R^2)},\quad 
\end{equation} 
where $(p,q)\in\Lambda^S$ with
\begin{equation} \label{eqS:admissible}
 \Lambda^S:=\left\{(p,q)\in[2,+\infty]^2: \frac{2}{p} + \frac{2}{q} = 1, \quad (p,q)\neq (2,\infty)\right\}.
\end{equation}
\end{theorem}

\begin{remark}
\label{rm:pole}
The non-collinear assumption on the poles may be dropped. However, we will not discuss it further, as it is not the main focus of the paper. Note that this geometrical assumption guarantees purely diffractive propagation among any three consecutive diffractions if we consider the propagator $e^{itP}$ or $\sin (t \sqrt{P})/\sqrt{P}$. See \cite{Yang22,Yang23} for more details, and \cite{FHH} for wave propagation in the collinear case. 
\end{remark}

\begin{remark} 
For the Aharonov--Bohm Hamiltonian with a single pole, i.e. $N=1$, the Strichartz estimates can be established using the dispersive estimate originally proved in \cite[Theorem 1.1]{FFFP} (or see \cite[Corollary 2.1]{GYZZ}) and Keel--Tao's abstract argument \cite{KT}.  
\end{remark}

The proof of Theorem \ref{thm:stri-S} is inspired by perturbation arguments (e.g., see \cite{Burq,BPST,BPSS}), with the help of the framework established in \cite{FZZ, Yang22}. For Schr\"odinger equations, to treat the Aharonov--Bohm Hamiltonian with multiple poles $P$ as a perturbation of the Hamiltonian $P_{\beta}$ with a single pole requires a global-in-time local smoothing estimate
$$
\big\| \chi e^{itP}f\big\|_{L_{t}^2(\R;{\dom}^{\frac12})}\leq C \|f\|_{L^2(\R^2)},
$$
where $\chi\in\CI_c$ is a smooth cutoff function with compact support. Here, the reference operator is given by
\begin{equation}
P_{{\beta}}:=\big(-i\nabla-\beta{\A}_{0}\big)^2
\end{equation}
where $\beta=\sum_{i=1}^N \alpha_i$ is the total flux. In particular, if $\beta=0$, $P_{\beta}$ is taken to be the free Laplacian on $\R^2$. 

\vspace{-2mm}

\subsection*{Related Works}

The study of decay and Strichartz estimates for dispersive equations has a rather long history, due to their central importance in both analysis and the theory of partial differential equations. We refer to \cite{BPSS, BPST, CDY2, CF, CS, CYZ, DF, DFVV, EGS1, EGS2, FV, FFT, FGK, GK, S} and the references therein for those estimates for Schr\"odinger and wave equations with electromagnetic potentials. Other related Strichartz results in nontrapping, rough-coefficient, asymptotically flat, conic, hyperbolic and exterior-domain settings include \cite{Burq,T00,MMT,ZZ2,BFM,BMW,EG,HTW,HZ,HSTZ,RS} and the references therein. Even for the single-pole Aharonov--Bohm Hamiltonian, i.e.~$N=1$, the situation becomes complicated. This is because the scaling-critical magnetic potential induces a long-range perturbation. Fanelli--Felli--Fontelos--Primo \cite{FFFP, FFFP1} studied the time-decay and Strichartz estimates for the Schr\"odinger equation with a single Aharonov--Bohm potential. More recently, Fanelli, Zheng and the second author \cite{FZZ}, as well as Gao, Yin, Zheng and the second author \cite{GYZZ} constructed kernels of the wave propagator and spectral measure, respectively, to prove the Strichartz estimate for the wave and Klein-Gordon equations. The success of the argument relies on the special structure of the operator $P$ when $N=1$.
For example, if $N=1$, one can translate the center to the origin and use the scaling invariance. Additionally, the eigenfunctions and eigenvalues of magnetic Laplacians on $\mathbb{S}^1$ are explicit, which allows one to construct the kernels. The Aharonov--Bohm Hamiltonian $P$ in \eqref{oper-P} with multiple solenoids no longer has rotation symmetry or scaling invariance. Hence, it requires new ingredients to establish the global-in-time Strichartz estimates. In similar settings, Strichartz estimates for the Schr\"odinger equations with multiple inverse-square potentials have been proved by Duyckaerts \cite{Duykaerts}, but only local in time. 

On the other hand, since the original work of Aharonov and Bohm \cite{AB59},  multipole Aharonov--Bohm Hamiltonians have recently drawn significant attention in spectral theory, scattering theory, and physics. Recent work on the multipole Aharonov--Bohm Hamiltonian includes self-adjoint extensions and spectral properties \cite{CF23}, variation of eigenvalues \cite{AF15, AF16,AFNN,Lena}, resolvents \cite{CDY,Sto89}, diffraction \cite{IT06,Yang21}, wave-traces \cite{Yang22}, and resonances \cite{AT11,AT14,Yang23}. Among these,  {\v{S}}t'ov{\'\i}{\v{c}}ek \cite{Sto89} investigated the case of two solenoids using the universal covering of the punctured plane and extended the study to finitely many solenoids. However, the results presented there involve an infinite sum, making it challenging to analyze decay or Strichartz estimates. Ito--Tamura \cite{IT06} studied the (semiclassical) scattering amplitudes for two solenoids, particularly in the asymptotic regime where their distance tends to infinity. Alexandrova--Tamura \cite{AT11,AT14} studied scattering resonances of two solenoids using a modified complex scaling method. They also computed resonances located within a logarithmic neighborhood at high energy. Eigenvalues of Aharonov--Bohm operators with varying poles have been studied by Abatangelo--Felli \cite{AF15, AF16}, Abatangelo--Felli--Noris--Nys \cite{AFNN}, etc.

\vspace{-2mm}

\subsection*{Acknowledgment}
We would like to thank Daniel Tataru and Jared Wunsch for their helpful discussions and interest in this project. MY acknowledges the support of the AMS-Simons travel grant and NSF DMS-2554813. JZ was supported by the Beijing Natural Science Foundation (1242011) and the National Natural Science Foundation of China (12531005).

\vspace{-2mm}

\subsection*{List of Notation}
\begin{itemize}
    \item The set of $N$ poles: $S = \{s_i=(x_i,y_i)\in\R^2: 1\leq i\leq N\}$.
    \item The bracket notation $\langle \lambda \rangle \coloneqq (1+|\lambda|^2)^{1/2}$.
    \item The Aharonov--Bohm Hamiltonian $P = (-i\nabla-\Vec{A})^2$ on $X=\mathbb{R}^2\backslash S$ with Friedrichs domain $\mathcal{D}$, where $\Vec{A}$ is given in equation \eqref{oper-P}.
      \item The Aharonov--Bohm Hamiltonian with a single pole at $s_k$ is given by 
    $P_{\alpha_k} := (-i\nabla-\alpha_k \vec A_0(x-x_k,y-y_k))^2, \ k = 1, \cdots, N$ on $\mathbb{R}^2\backslash \{s_k\}$ with Friedrichs domain $\mathcal{D}_{k}$.
    \item The model operator is given by $P_{\beta}= (-i\nabla-\beta\Vec{A_0})^2$ on $\mathbb{R}^2\backslash \{0\}$ with Friedrichs domain $\mathcal{D}_{\beta}$, where  $\beta=\sum_{i=1}^N \alpha_i$. This is the free Hamiltonian only when $\beta=0$. 
    \item The resolvent of $P$ is given by $R(\lambda)\coloneqq (P-\lambda^2)^{-1}$ on 
    $\R^2\setminus S$.
    \item The model resolvent is given by $R_{\beta}(\lambda) \coloneqq (P_{\beta}-\lambda^2)^{-1}$ on $\mathbb{R}^2\backslash \{0\}$.
\end{itemize}

\section{Local smoothing estimates}
\label{section:locsmoothing}
In this section, we prove local smoothing estimates for Schrödinger equations by combining the results of \cite{CDY} on low energy resolvents and \cite{Yang22} on high energy resolvents. 

\subsection{Operators and domains}\label{sec:dom}
We first recall some preliminaries about the operator $P$ and its domains, which were stated in \cite{Yang21}. The domains of $P_{\beta}$ are special cases of the domains of $P$. We define the power domains 
\begin{equation}
    \mathcal{D}^s:=\big\{u\in L^2:  (P+1)^{\frac s2}u \in L^2(\R^2) \big\}
\end{equation}
where $(P+1)^{\frac s2}$ is defined using the functional calculus. We remark that, away from the set of poles $S$, the domain $\mathcal{D}^s$ agrees with the usual Sobolev space $ H^s(\R^2)$. When $s\geq 0$, another equivalent way is to define the space
\begin{equation}
    \mathcal{D}^s:=\big\{u\in L^2:  P^{\frac s2}u \in L^2(\R^2) \big\}
\end{equation}
because of the inequality
$$ \norm[L^2]{P^{s/2} u} \leq \norm[L^2]{(1+P)^{s/2} u} \leq C_s(\norm[L^2]{u} + \norm[L^2]{P^{s/2}u}).$$ 
Similarly, we define the power domains for the single-pole Hamiltonian 
\begin{equation}
    \mathcal{D}^s_{\alpha_k}:=\big\{u\in L^2:  (P_{\alpha_k}+1)^{\frac s2}u\in L^2(\R^2) \big\}
\end{equation}
Note that the Friedrichs domain $\dom$ of $P$ agrees with the power domain $\dom^2$ in the above notation. 

\subsection{Resolvent estimates}
\label{sec:res-est}

We first prove a uniform resolvent estimate near the real axis.
\begin{proposition}\label{est:resovent} 
Let $\chi\in C_c^\infty(\R^2)$. Then there exists $\epsilon_0>0$ and a constant $C$ independent of $\lambda$ such that 
\begin{gather}
\label{resolvent-est1}
\norm[L^2 \to \dom^j ]{\chi R(\lambda) \chi} \leq C (1+{|\lambda|})^{-1+j},\ j =0,1.
\end{gather}
where $\lambda\in\C$ satisfying $|\mathrm{Im}\,\lambda|\leq \epsilon_0$. 
\end{proposition}

\begin{proof} 
High energy resolvent estimates in \cite[Theorem 5.2]{Yang22} yield that for any small $\epsilon > 0$
\[\norm[L^2 \to \dom^j ]{\chi R(\lambda) \chi} \leq C |\lambda|^{-1+j},\ j=0,1\]
in the region $\{\Re\lambda >M, |\Im\lambda| \leq \epsilon\}$ for some $M > 0$. The low-energy uniform resolvent asymptotics in \cite[Theorem 2,3]{CDY} yield that
\[\norm[L^2 \to \dom^j ]{\chi R(\lambda) \chi} \leq C,\ j=0,1 \]
in a small neighborhood $\{|\lambda|< \epsilon'\}$ for some $\epsilon'>0$. 
Combining them with the absence of resonance or embedded eigenvalue on the essential spectrum $[0,+\infty)$, which yields the resolvent estimate \eqref{resolvent-est1} for intermediate frequencies (cf.~\cite[Theorem 4.17]{DZ}), we conclude the proof. 
\end{proof}

By Proposition \ref{est:resovent}, we have the following
\begin{corollary}
\label{cor:resolvent} Let $\chi\in C_c^\infty(\R^2)$ and $s\in \R$. Then there exists $\epsilon_0>0$ and a constant $C$ independent of $\lambda$ such that 
\begin{gather}
\label{resolvent-est1'}
\|\chi R(\lambda) \chi\|_{\mathcal{D}^{s-1}\to \mathcal{D}^{s}} \leq C,
\end{gather}
where $\lambda\in\C$ satisfying $|\mathrm{Im}\,\lambda|\leq \epsilon_0$. 
\end{corollary}

\begin{proof}
Using duality and interpolation, it suffices to prove the corollary for $s\in\NN$. We first prove \eqref{resolvent-est1'}. By the dual estimate of \eqref{resolvent-est1} for $j=1$, we have 
\[\|\chi R(\lambda) \chi\|_{\mathcal{D}^{-1}\to L^2} \leq C,\]
which shows that \eqref{resolvent-est1'} holds for $s=0, 1$.
Now we prove \eqref{resolvent-est1'} inductively in $s$. Assume \eqref{resolvent-est1'} holds for all $s\leq m$. For $s=m+1$, we only need to show
\[\|P\chi R(\lambda) \chi\|_{\mathcal{D}^{m}\to \dom^{m-1}} \leq C.\]
By the resolvent identity  
\[P\chi R(\lambda) \chi = [P,\chi] R(\lambda) \chi + \chi R(\lambda)[P,\chi] + \chi R(\lambda) \chi P,\]
and the mapping properties of $[P,\chi]$ and $P$, 
it is easy to see that each of the three terms is bounded from $\mathcal{D}^m$ to $\dom^{m-1}$. Indeed, we have
\[\begin{split}\norm[\mathcal{D}^{m}\to \dom^{m-1}]{[P,\chi] R(\lambda) \chi} &\leq \norm[\mathcal{D}^{m}\to \dom^{m}]{ \widetilde\chi R(\lambda) \chi} \norm[\mathcal{D}^{m}\to \dom^{m-1}]{[P,\chi] }
\\&\leq \norm[\dom^{m-1}\to \dom^m]{ \widetilde\chi R(\lambda) \chi} \norm[\mathcal{D}^{m}\to \dom^{m-1}]{[P,\chi] }\leq C,
\end{split}\]
where we take $\widetilde{\chi}\equiv 1$ on $\supp \chi$ and use the continuous embedding $\mathcal D^m\hookrightarrow\mathcal D^{m-1}$. Similarly, we obtain
\begin{gather*}
\norm[\dom^{m}\to \dom^{m-1}]{\chi R(\lambda)[P,\chi]} \leq \norm[\dom^{m-1}\to \dom^{m-1}]{\chi R(\lambda)\widetilde\chi} \norm[\dom^{m}\to \dom^{m-1}]{[P,\chi]}\leq C,\\
\norm[\dom^{m}\to \dom^{m-1}]{\chi R(\lambda)\chi P} \leq \norm[\dom^{m-2}\to \dom^{m-1}]{\chi R(\lambda)\chi} \norm[\dom^{m}\to \dom^{m-2}]{P}\leq C.  
\end{gather*}
This finishes the proof.
\end{proof}

\subsection{Local smoothing estimates}
\label{sec:locS}
We now prove global-in-time local smoothing estimates for Schr\"odinger equations.
\begin{theorem}
\label{thm:local-smoothing_global} 
If $u$ is a solution to the Schr{\"o}dinger equation 
\begin{equation*}
\begin{cases}
     (D_t - P) u = 0, \\
    u(0,x) = f(x),
\end{cases}
\end{equation*}
and $\chi \in C_c^{\infty}(\RR^2)$ is a compactly supported smooth function, then $u$ satisfies a local smoothing estimate
  \begin{equation}\label{loc:S}
    \norm[L^{2}(I; \dom^{\frac{1}{2}})]{\chi u} \leq C \norm[L^{2}(\RR^2)]{f},
  \end{equation}
where $I$ is either $[0,T]$ or $\RR$ respectively and the constant $C$ is independent of $I$.  
\end{theorem}

\begin{remark}
Since the constant $C$ in \eqref{loc:S} is independent of $I$, the estimate \eqref{loc:S} is global-in-time. The local-in-time local smoothing estimate for the Aharonov--Bohm Hamiltonian with multiple poles is proved in the second author's previous work \cite{Yang22}. Also note that Theorem~\ref{thm:local-smoothing_global} implies the dual estimate
  \begin{equation}\label{dualsmoothing}
    \norm[L^{2}]{\int _{I}e^{-isP}\chi F(s) \,ds} \leq
    C \norm[L^{2}(I; \dom^{-\frac{1}{2}})]{F},
  \end{equation}
  for $I$ being either $[0,T]$ or $\RR$.
\end{remark}

\begin{proof}
We only consider the case $I=\R$. We use the resolvent estimate \eqref{resolvent-est1} and the standard $TT^*$-argument. We define
    \[Tu = \chi e^{itP} u.\]
    We now show $T$ is a bounded operator from $L^{2}(X)$ to $L^{2}(\reals;\dom^{\frac{1}{2}})$. It suffices to show that $TT^{*}$ is bounded from $L^{2}(\reals ; \dom^{-\frac{1}{2}})$ to $L^{2}(\reals; \dom^{\frac{1}{2}})$.  The operator $TT^{*}$ is given by
  \begin{align*}
    TT^{*}F &= \chi\int_{\reals}e^{i(t-s)P}\chi F(s)\,ds \\
    &= \chi\int_{s < t}e^{i(t-s)P}\chi F(s)\,ds + \chi \int_{s >
      t}e^{i(t-s)P }\chi F(s)\,ds =: \chi T_{+}F+ \chi T_{-}F.
  \end{align*}  
Note that by Duhamel's principle, $T_{\pm}F$ are both solutions of the inhomogeneous Schr{\"o}dinger equation
  \begin{equation}
  \label{eq:inhomogeous}
  \begin{cases}
    D_{t} u - P u =  {\pm} \frac{1}{i}\chi F,\\ 
    u(t)\big|_{t=\pm\infty}=0.
    \end{cases}
  \end{equation}
Suppose now that $F$ is compactly supported in time, i.e., $F(t,x) = 0$ for $t \notin [-t_{0},t_{0}]$ for some $t_0$.  In this case, $T_{+}F$
vanishes for $t < -t_{0}$ and $T_{-}F$ vanishes for $t >t_{0}$. 

For this $F$ compactly supported in time, we wish to show that there is a constant $C$, independent of $t_{0}$, such that 
  \begin{equation*}
    \int _{\reals}\norm[\dom^{\frac{1}{2}}]{\chi T_{\pm}F(t,x)}^{2} \,dt
    \leq C \int_{\reals}\norm[\dom^{-\frac{1}{2}}]{F(t,x)}^{2}\,dt.
  \end{equation*}
By Plancherel's theorem, it suffices to show that
  \begin{equation}\label{loc:s'}
    \int_{\reals}\norm[\dom^{\frac{1}{2}}]{\chi
      \widehat{T_{\pm}F}(E, x)}^{2}\,dE \leq C
    \int_{\reals}\norm[\dom^{-\frac{1}{2}}]{\hat{F}(E,x)}^{2}\,dE, 
  \end{equation}
where $\hat{F}$ denotes the Fourier transform of $F$ in $t$. Taking the Fourier transform of the equation \eqref{eq:inhomogeous}, we see that $\widehat{T_{\pm}F}(E,x)$ solves the equation 
  \begin{equation*}
    (P - E) \widehat{T_{\pm}F} =  {\mp} \frac{1}{i}\chi \hat{F}.
  \end{equation*}
Moreover, the condition on the support of $F$ implies that $\widehat{T_{+}F}$ is holomorphic in the lower half-plane, while $\widehat{T_{-}F}$ is holomorphic in the upper half-plane. In particular, write $\widehat{R}(z) = (P - z)^{-1}$ where it is invertible,
  \begin{equation*}
    \widehat{T_{\pm}F}(E,x) = \lim _{\epsilon \downarrow 0} \widehat{R}(E \mp i\epsilon) \left(  {\mp} \frac{1}{i}\chi \hat{F}(E,x)\right).
  \end{equation*}
We now apply the uniform resolvent estimate \eqref{resolvent-est1'} in Corollary \ref{cor:resolvent} with $s=1/2$, which shows that there is a constant $C$ independent of $E$ and $t_{0}$, so that
  \begin{equation*}
    \norm[\dom^{\frac{1}{2}}]{\chi \widehat{T_{\pm}F}(E,x)}^{2} \leq C \norm[\dom^{-\frac{1}{2}}]{\hat{F}(E,x)}^{2}.
  \end{equation*}
Integrating in $E$ then finishes the proof of \eqref{loc:s'} for $F$ compactly supported in time.  For the $F$ in the general setting, we simply note that the constant is independent of the support and that compactly supported functions are dense in $L^{2}_{t}$.
\end{proof}

\section{Proof of Strichartz estimates}

\subsection{Strichartz estimates for AB Hamiltonian with a single pole}

We first briefly recall Strichartz estimates for the Aharonov--Bohm Hamiltonian with a single pole, as these are key ingredients in the proof of our main theorem.
\begin{proposition}\label{prop:stri-one} Let $P_{\alpha}=(-i\nabla-\alpha\A_0)^2$ be the single-pole Aharonov--Bohm Hamiltonian. For $(p,q)\in\Lambda^S$ and $I$ being an interval or $\RR$, the solution of Schr\"odinger equations
\begin{equation*}
\begin{cases}
(D_t-P_{\al})u=0,\\
u|_{t=0}=f(x),
\end{cases}
\end{equation*}
satisfies the Strichartz estimates
\begin{equation}\label{str-S-j}
\|u(t,x)\|_{L^p_t(I;L^q(\R^2))}\leq C \|f\|_{L^2(\R^2)}.
\end{equation}
\end{proposition}

\begin{proof} The Strichartz estimate \eqref{str-S-j} follows from the dispersive estimates proved in \cite[Theorem 1.1]{FFFP} (or see \cite[Corollary 2.1]{GYZZ}) and Keel--Tao's abstract method in \cite{KT}. 
\end{proof}

We now construct a partition of unity $\{\chi_j\}_{j=0}^{N}$ of $\R^2$ with respect to poles of the Hamiltonian $P$. For each pole $s_j$, we define $U_j$ to be a simply connected neighborhood of $s_j$ in $\R^2$ such that $U_j$ does not contain any $s_k$ when $k\neq j$, and $\{U_j\}_{j=1}^N$ is an open cover of $U_0\coloneqq B(0,R_0)$ for $R_0\gg 0$. Let $\chi_0, \chi_1,\cdots,\chi_N$ be a partition of unity subordinate to the cover $\{U_j\}_{j=0}^N$ of $\R^2$, where $U_0$ does not contain any pole, with $\chi_j\equiv 1$ near $s_j$ for $1\leq j \leq N$.  Note that by this construction $\supp\chi_j, j\neq 0$ only contains the pole $s_j$ but no other poles. 

We define the Aharonov--Bohm Hamiltonian with a single pole at $s_j$
\begin{equation}
    \label{eq:Pj}
    P_j := (-i\nabla-\alpha_j\A_0(x-x_j, y-y_j))^2, \ j = 1, \cdots, N
\end{equation}
and $P_0 := P_{\beta}$ with the corresponding domain $\dom_j$. The corresponding vector potentials are denoted by $\A_j$. We first prove the following
\begin{lem}
\label{lem:conj}
    There exists $\varphi_j\in\CI(\R^2)$ such that, on $U_j$
    \begin{equation}
    \label{eq:conj}
        P=e^{-i\varphi_j} P_j e^{i\varphi_j}, \ j=0,1,\cdots N,
    \end{equation}
in the sense that $P u = e^{-i\varphi_j} P_j e^{i\varphi_j} u$ for any $u\in \dom$ supported in $U_j$. 
\end{lem}

\begin{proof}
We take $\tilde{\chi}_j\in\mathcal{C}_c^\infty(\R^2)$ to be a smooth cut-off function such that $\tilde{\chi}_j\equiv 1$ on $U_j$, and is supported in a slightly larger simply connected open set $V_j\supset U_j$, such that only the pole $s_j\in\supp\tilde{\chi}_j$ where $j=1\cdots, N$. Similarly, we can define $\tilde{\chi}_0\in\mathcal{C}^\infty(\R^2)$ such that $\tilde{\chi}_0\equiv 1$ on $U_0$ with no pole in $\supp\tilde{\chi}_0$. Hence, there exist $\psi_j\in \CI(V_j)$ satisfying 
\[
    \nabla\psi_j=\A_j-\A \qquad\text{on }V_j .
\]
Now we set 
$\varphi_j=\widetilde\chi_j\psi_j\in\CI(\R^2)$. Hence, on each $U_j$ where $\tilde{\chi}_j\equiv 1$, we have $\nabla\varphi_j=\A_j-\A$. Therefore
\[
    e^{i\varphi_j}(-i\nabla-\A)e^{-i\varphi_j} =-i\nabla-\A_j \qquad\text{on }U_j
\]
and squaring the above equation gives $e^{i\varphi_j}Pe^{-i\varphi_j}=P_j$ on $U_j$. 
\end{proof}

For notational convenience, we write $\tP_j := e^{i\varphi_j} P e^{-i\varphi_j}$, for $j=0,1,\cdots N.$ Hence, we have $\tP_j = P_j$ over $U_j$, for functions locally in $\dom_j$ supported in $U_j$.   

\subsection{Strichartz estimates for Schr\"odinger equations}

We are now ready to prove our main theorem. We follow the idea of \cite{BMW} and use, as a new ingredient, the local gauge transform in Lemma \ref{lem:conj} to reduce the multipole problem to single-pole ones.

\begin{proof}[Proof of Theorem \ref{thm:stri-S}]
Let $u$ be the solution of the homogeneous Schr\"odinger equation with initial data $f$, i.e.,
\begin{equation}
\begin{cases}
(D_t - P)u=0,\\
u|_{t=0}=f(x).
\end{cases}
\end{equation}
 Set $u_j:=\chi_j u$, 
then $u_j$ solves the following inhomogeneous Schr\"odinger equation on $X$
\begin{equation}
\begin{cases}
(D_t - P)u_j=-[P,\chi_j]u,\\
u_j|_{t=0}=\chi_j f(x),
\end{cases}
\end{equation}
where $[P,\chi_j]$ is a first order differential operator supported where $d\chi_j\neq0$, hence away from all poles by construction. Taking the unitary conjugation by $e^{i\varphi_j}$ and writing $\tilde{u}_j := e^{i\varphi_j}u_j$, the equation becomes 
\begin{equation}
\begin{cases}
(D_t - \tP_j)\tilde{u}_j = - e^{i\varphi_j} [P,\chi_j]u,\\ 
\tilde{u}_j|_{t=0}=e^{i\varphi_j}\chi_j f(x).
\end{cases}
\end{equation}
As the support of $\tilde{u}_j, [P,\chi_j]u$ and $\chi_j f(x)$ are all contained in $U_j$, by Lemma \ref{lem:conj}, the equation is the same as solving the inhomogeneous equation of the Aharonov--Bohm Hamiltonian $P_j$ with a single pole at $s_j$:
\begin{equation}
\begin{cases}
(D_t - P_j)\tilde{u}_j = - e^{i\varphi_j} [P,\chi_j]u,\\ 
\tilde{u}_j|_{t=0}=e^{i\varphi_j}\chi_j f(x).
\end{cases}
\end{equation}

We write $\Tilde{u}_{j} = u_{j}' + u_{j}''$, where $u_{j}'$ is the solution
of the homogeneous equation with the same initial data
\begin{equation}
\begin{cases}
(D_t - P_j)u'_j = 0,\\
u'_j|_{t=0}=e^{i\varphi_j} \chi_j f(x),
\end{cases}
\end{equation}
and $u_{j}''$ is the solution of the inhomogeneous equation with
zero initial data. 

We know by Proposition \ref{prop:stri-one} and local domain equivalence that $u'_{j}$ satisfies the homogeneous Strichartz estimate
\begin{equation*}
\|u'_j(t,x)\|_{L^p_t(I;L^q(\R^2))}\leq C \|f\|_{L^2}.
\end{equation*}

We now set $v_{j}(t,x) = -e^{i\varphi_j} [P, \chi_{j}]u.$  Then by Duhamel's principle,
\begin{equation*}
  u_{j}'' =  {i} \int_{0}^{t}e^{i(t-s)P_j} v_{j}(s) \, ds. 
\end{equation*}
Note that $[P, \chi_{j}]$ is a compactly supported differential
operator of order one supported away from the poles and thus the local smoothing estimate Theorem \ref{thm:local-smoothing_global} implies that there is a constant
$C$ so that
\begin{equation}
\label{eq:aux1}
  \norm[L^{2}({\RR}; \dom_j^{-\frac{1}{2}})]{v_{j}} \leq C\norm[L^{2}]{f}.
\end{equation}
We wish to show that $u_{j}''$ obeys the Strichartz estimates.  As
we are assuming that $p > 2$ in condition \eqref{eqS:admissible}, the Christ--Kiselev
lemma~\cite{CK} implies that it suffices to show the Strichartz 
estimate for 
\begin{equation*}
   \tilde{u}_{j}'':=  {i} \int_{\RR}e^{i(t-s)P_j}v_{j}(s) \, ds =  {i} \, e^{itP_j} \int_{\RR}e^{-isP_j}v_{j}(s) \, ds.
\end{equation*}
By the dual local smoothing estimate for the Hamiltonian $P_j$ with a single pole in Theorem~\ref{thm:local-smoothing_global} and the estimate \eqref{eq:aux1}, 
\begin{equation}
\label{eq:aux2}
  \norm[L^{2}]{\int_{\RR}e^{-isP_j} v_{j}(s) \,ds} \leq
  C\norm[L^{2}({\RR};\dom_{j}^{-\frac{1}{2}})]{v_{j}}  \leq C\norm[L^{2}]{f}.
\end{equation}
This together with the homogeneous Strichartz estimate \eqref{str-S-j} for the propagator $e^{itP_j}$ shows
\begin{equation}
   \|\tilde{u}_{j}''\|_{L^p_t(\R;L^q(\RR^2))}\leq C \norm[L^{2}]{\int_{\RR}e^{-isP_j} v_{j}(s)\, ds} \leq C\norm[L^{2}]{f},
   \end{equation}
where we use \eqref{eq:aux2} for the last inequality. Therefore, combining estimates for $u_j'$ and $u_j''$, and summing over all $j$'s, we complete the proof.  
\end{proof}

\begin{center}

\end{center}

\end{document}